\documentclass[preprint 11pt]{elsarticle}

\usepackage{syntonly}
\usepackage{verbatim}

\usepackage{amsthm,amsmath}

\usepackage{yfonts}
\usepackage{amssymb}
\usepackage{mathtools}

\usepackage{mathrsfs}

\usepackage{hyperref}

\usepackage{xcolor}

\usepackage{caption}
\usepackage{subcaption}

\newtheorem{definition}{Definition}
\newtheorem{assumption}{Assumption}
\newtheorem{lemma}{Lemma}
\newtheorem{theorem}{Theorem}

\theoremstyle{remark}
\newtheorem{remark}{Remark}

\def\eps{\varepsilon}

\newcommand\norm[1]{\left\lVert#1\right\rVert}

\title{On filter-type estimation of discretely sampled cyclic
long-memory processes.}

\author[1]{Antoine Ayache} 
\ead{antoine.ayache@univ-lille.fr}

\author[2]{Serhii Kravchenko}
\ead{20522507@students.latrobe.edu.au}

\author[2]{Andriy Olenko\corref{cor1}} 
\ead{a.olenko@latrobe.edu.au}

\cortext[cor1]{Corresponding author. Email address: a.olenko@latrobe.edu.au}

\affiliation[1]{organization={CNRS UMR 8524 - Laboratoire Paul-Painlevé,  Université de Lille}, 
city={Lille},
postcode={F-59000}, 
country={France}}
					 
\affiliation[2]{organization={School of Computing, Engineering and Mathematical Sciences, La Trobe University}, city={Melbourne},
					postcode={3086},
					country={Australia}}

\begin{document}
\allowdisplaybreaks

	\begin{abstract}
		The generalized filtered method of moments was developed in the recent papers by Alomari et al., 2020, and Ayache et al., 2022. It used functional data obtained from continuously sampled cyclic long-memory stochastic processes to simultaneously estimate their parameters.  However, the majority of applications deal with discretely sampled processes or time series. This paper extends the approach to accommodate discrete-time scenarios. It proves that the new discrete estimates exhibit analogous properties to the continuous case and are strongly consistent with the same rates of convergence.  The numerical study results are presented to illustrate the theoretical findings and to indicate the sampling rates and resolution levels required for accurate estimates.
	\end{abstract}

		\maketitle

	\section{Introduction}	
	
	The majority of applied time series publications associate the phenomenon of long-range dependence to
	singularities of spectral densities at the origin, see~\cite{LeonOle:2013} and references therein. However, singularities at non-zero frequencies
	play an important role in investigating the cyclic or seasonal behaviours of
	time series which can exhibit long-memory properties as well. In contrast to seasonal time series, non-seasonal cycles are often unknown in advance, presenting a challenge, in particular, in the analysis of physical and financial data.
	
	Several parametric and semi-parametric methods were proposed to model and study the case of time series with non-zero spectral poles, see \cite{ArtRob:2000, Giraitis:2001, Hidalgo:2005, Whitcher:2004, Kechagias:2024} and the references therein. Most of such publications either considered models with specific parametric forms of spectral densities at their singularity points or fractional processes of the Gegenbauer type, see the recent review paper \cite{Peiris:2018}. To estimate their parameters a number of approaches based on Gaussian maximum likelihood or quasi-likelihood were developed in the econometrics literature, see, for example,  \cite{ArtRob:2000, Giraitis:2001, Hosoya:1997}. 
	An exact maximum-likelihood estimator for the autoregressive fractional integrated moving average (ARFIMA) process was developed in \cite{Sowell:1992}. The procedure was computationally expensive for large datasets and required approximations of the maximum likelihood, see \cite{Jensen:1999}. Wavelet approximation has been used for the maximum likelihood parameter estimation as an alternative to the exact estimator or other approximation techniques such as the Fourier transform. In \cite{McCoy:1996} the discrete wavelet transform was applied to the ARFIMA process. Their results were extended in \cite{Jensen:1999} to simultaneously estimate short- and long-memory parameters of the ARFIMA model. The simulation studies revealed the robustness of wavelet approximation of the maximum likelihood compared to Whittle's maximum likelihood. 
	Later, the paper \cite{Whitcher:2004} discussed an approach to estimating parameters for the wider class of Gegenbauer processes. These results were extended to  $k-$frequency GARMA models in \cite{Boubaker:2015}. 
	There were also attempts to apply minimum contrast estimation methodology for long-range dependent 
	models, see \cite{Anh:2004, Espejo:2015}.   
	Unfortunately, the mentioned results either do not provide simultaneous estimators of cyclic and long memory parameters or exhibit a lack of robustness leading to inconsistent parameter estimates for misspecified models.
 
	Due to its significance in practical applications, this research area remains active, with the introduction of new models and estimation approaches. For example, some other directions in the statistical inference of random processes 
	characterized by certain singular properties of their spectral densities were investigated in \cite{Kechagias:2024, Leonenko:1999, Tsai:2015, Sabzikar:2019}. Bayesian approaches and the comparison of alternative estimation methods can be found in \cite{Hunt:2022, Hunt:2023, Smallwood:2021, Whitcher:2004}.
	
	To avoid repetitions, we refer the readers to very detailed motivation, discussion and various examples in \cite{AAFO, Olenko:2022, Olenko:2013}.
	
	The publications \cite{AAFO, Olenko:2022} proposed a new approach to simultaneously estimate cyclic and long-memory parameters that was based on filter transforms (including wavelet transforms as a special case) of stochastic processes. The suggested estimates were developed for functional data obtained from continuously sampled stochastic processes.  However, in the majority of applications only processes or time series that are discretely sampled on a finite interval are available. We demonstrate how to extend the approach from \cite{AAFO, Olenko:2022} to discrete-time sampling scenarios. The suggested new discrete estimates exhibit analogous properties and converge to the true values of the parameters as in the continuous case. Moreover, it is shown that they have the same rate of convergences.
	
	The proofs are based on an examination of the proximity between discrete and continuous filter transforms. It is shown that the deviations between them are of a smaller order compared to the distances between continuous estimators and the true values of parameters. Thus, it can be applied to extend the results in \cite{AAFO} from the continuous case to discretely sampled data. These findings regarding the closeness of discrete and continuous filter transforms can also be of interest in other statistical applications that use filtering.
	
	The paper is structured as follows. Section~\ref{def_assumptions} provides the main definitions and notations. Section~\ref{section 1st statistic} introduces a version of filter transforms for the case of finite discrete samples. Then, it derives a series of results on the closeness of discrete and continuous filter transforms and uses them to investigate the convergence of the first discrete statistics.  Section~\ref{section 2nd statistic} introduces the second discrete statistics.  The adjusted statistics that simultaneously 
	estimate location and long memory parameters are considered in Section~\ref{Section Estimation}. 
	Section~\ref{sec_num} presents numerical studies that illustrate the obtained results. Finally, some future research directions are given in Section~\ref{concl}.

	\section{Definitions and assumptions}\label{def_assumptions}
	
	This section provides main definitions and assumptions about the considered class of semiparametric models and used filter transforms, see \cite{AAFO, Olenko:2022} for more details.
	
	The paper will use the following notations: $\mathbb{R}_+ = (0, +\infty)$ and,  unless otherwise specified, $|| \cdot ||$ will  denote the $L_2$ norm in a considered space. The symbols $C,$ $\varepsilon$ and their versions with subindices will denote constants which are not important for our exposition. Moreover, the same symbol may be used for different constants appearing in the same proof.
	
	Let $ X(t),$ $t \in \mathbb{R},$ be a measurable mean-square continuous real-valued stationary zero-mean
	Gaussian stochastic process defined
	on a probability space $(\Omega, \mathcal{F}, P),$ with the
	covariance function
	\[\mathrm{B} (r):={\rm Cov}(X(t), X(t'))=\int_ \mathbb{R} \mathrm{e}^{iu(t-t')} F(du), \quad t, t' \in \mathbb{R}, \]
	where $r=|t-t'|$ and $F(\cdot)$ is a non-negative finite symmetric measure on $\mathbb{R}.$
	
	For each $t,$ the random variable $X(t)$ belongs to the  space $\mathscr{L}_2 = \mathscr{L}_2(\Omega, \mathcal{F}, P)$ of random variables  with a finite second moment, i.e.  $\mathbb{E}|X(t)|^2 < \infty.$
	The norm in $\mathscr{L}_2$ is defined by $|| X(t) || = \sqrt{\mathbb{E}|X(t)|^2}.$

	\begin{definition}
		The real-valued random process $X(t),$ $t \in \mathbb{R},$ possesses an absolutely
		continuous spectrum if there exists a non-negative even function $f(\cdot)\in L_{1}( \mathbb{R})$ such that
		\[F(u)=\int_{-\infty }^ {u} f(\lambda) d\lambda, \quad  u\in \mathbb{R}. \]
	\end{definition}
	
	The function $f(\cdot)$ is called the spectral density of the process $ X(t).$

	The real-valued process $ X(t), t \in \mathbb{R},$ with an absolutely continuous spectrum has the
	following isonormal spectral representation
	\[X(t)= \int_ {\mathbb{R}}\mathrm{e}^{i t\lambda} \sqrt{f(\lambda)} dW(\lambda),\]
	where $W(\cdot)$ is a complex-valued Gaussian orthogonal random measure defined on ~$\mathbb{R},$ which satisfies the condition
	$W\left(\left[\lambda_1,\lambda_2\right]\right)=W\left(\left[-\lambda_2,-\lambda_1\right]\right)$ for any $\lambda_2>\lambda_1>0,$ see  \cite[\S 6]{Taqqu:1979}.

	The following assumption in the spectral domain introduces the semi\-para\-met\-ric model investigated in this paper.
	\begin{assumption}\label{Assumption_1}
		Let the spectral density $f(\cdot)$ of $X(t)$ admit the following representation
		\[f(\lambda)=\frac{h(\lambda)}{|\lambda^{2}-s_{0}^{2}|^{2\alpha}},\quad \lambda \in \mathbb{R},\]
		where $s_{0}> 1,\, \alpha\in (0,{1}/{2})$ and $h(\cdot)$ is an even non-negative bounded function that is four times continuously differentiable. Assume that its derivatives of order $i$ satisfy $h^{(i)}(0)=0,$ $i=1,2,3,4.$ Also, $h(0)=1,$ $h(\cdot)>0$ in some neighborhood of $\lambda=\pm  s_0,$ and for all $\varepsilon>0$ it holds
		\[\int_\mathbb{R}\frac{h(\lambda)}{(1+|\lambda|)^\varepsilon }d\lambda<\infty.\]
	\end{assumption}
	Let us denote by \[c_1: = \max_{\lambda \in [0, \frac{1}{2}]}(h^{(2)}(\lambda), h^{(4)}(\lambda)).\]
	
	\begin{remark}
		Stochastic processes with spectral densities satisfying Assumption~\ref{Assumption_1} have cyclic (seasonal) long memory. Their spectral densities have singularities at non-zero locations $\pm s_0.$ The parameter $s_0$ determines seasonal or cyclic behaviour. The parameter $\alpha$ is a long-memory parameter. The boundedness of $h(\cdot)$ guarantees that the spectral density does not have singularities at other locations. Covariance functions of
		such processes have hyperbolically decaying oscillations and are non-integrable, that is $\int_{\mathbb{R}}|\mathrm{B} (r)|dr = \infty,$ when  $\alpha\in\left(0,1/2\right),$ see \rm \cite{AAFO, ArtRob:1999}.
		
	\end{remark}
	
	\begin{remark}
		An example of such processes is the Gegenbauer model which has spectral density satisfying Assumption~\ref{Assumption_1} since
		\[	f(\lambda) = C(2|\cos\lambda - \cos s_0|)^{-2\alpha}
		\sim  C|\lambda^2 - s_0^2|^{-2\alpha},
		\]
		when $\lambda \rightarrow \pm{s_0},$ see \rm\cite{Espejo:2015}. The plot of such density and the trajectory of the corresponding process are shown in Figure~\ref{fig1*}.
	\end{remark}
	
	\medskip
	
	\begin{remark}
		In this paper cyclic long-memory processes with spectral densities at nonzero frequencies $s_0$ are considered. Differences between the
		cases of spectral singularities at the origin and other locations were discussed in detail in \cite{AAFO, ArtRob:1999}.
		Without loss of generality, it will be assumed that $s_0 > 1.$ Indeed, the
		frequency $s_0$ can be computed as $s_0 = 1/T,$ where $T$ is a period of a time series.  By changing the time unit, the parameter $s_0$ can be made greater than~1.
	\end{remark}

	 \begin{figure}[ht]
			\begin{minipage}{0.5\textwidth}
					\includegraphics[width = \textwidth, height = \textwidth,trim={0 0 0 2cm},clip]{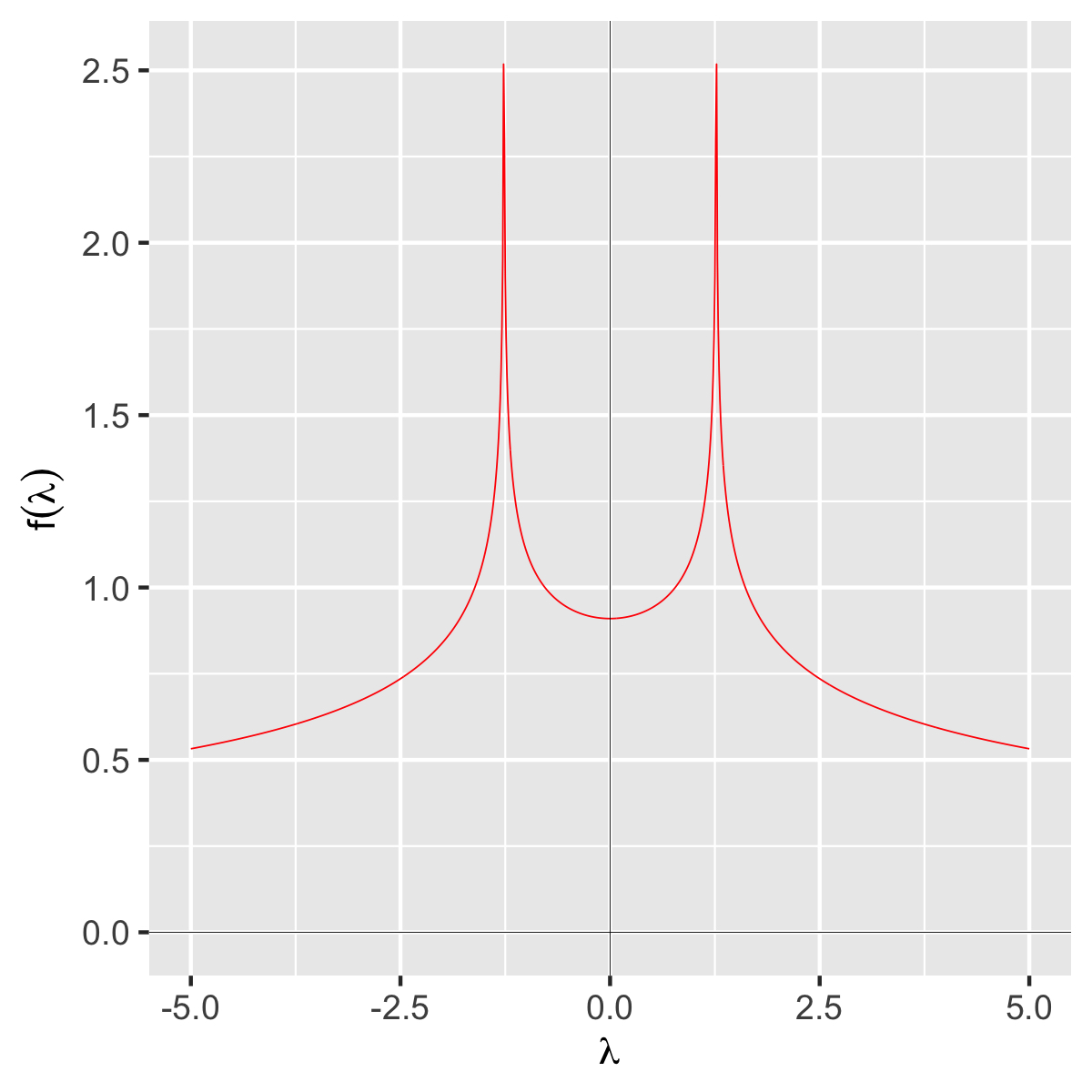}
			\end{minipage}
			\begin{minipage}{0.5\textwidth}
				\includegraphics[width = \textwidth, height = \textwidth,trim={0 0 0 2cm},clip]{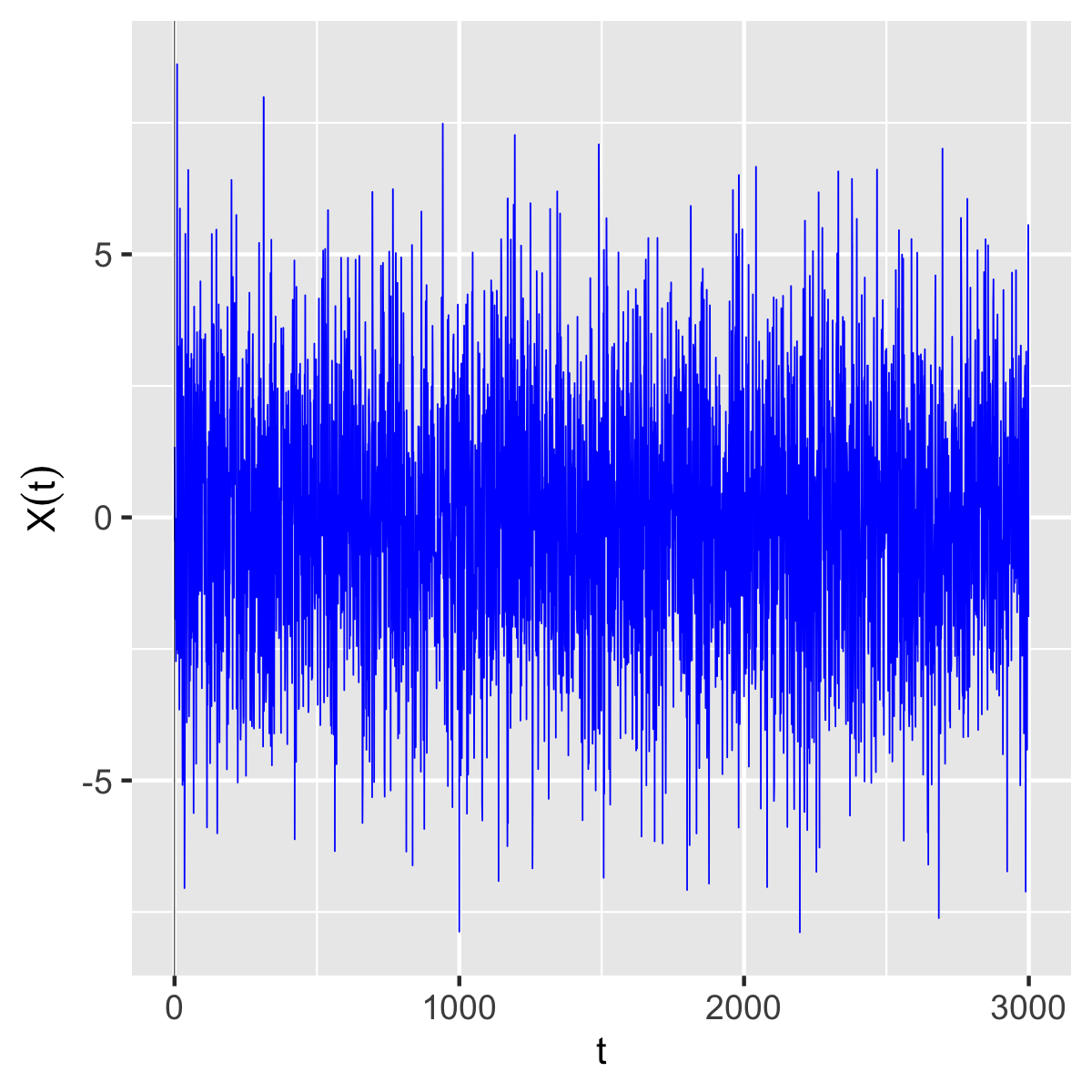}
			\end{minipage}		
		\caption{Plots of the spectral density(left) and the realization of the corresponding time series (right) for the Gegenbauer model with $s_0 = 1.27$ and $\alpha = 0.1$ }\label{fig1*}
	\end{figure}

	\medskip
	Let us consider a real-valued even function $\psi (\cdot)\in L_1(\mathbb{R})\cap L_2(\mathbb{R})$. 
	Its Fourier transform $\widehat{\psi}(\cdot)$ is defined, for each $\lambda \in \mathbb{R}$, as $\widehat{\psi}(\lambda)=\int_\mathbb{R}e^{-i\lambda t} \psi(t) dt$. Thus, $\widehat{\psi}(\cdot)$ is a bounded even real-valued function.
		
	\begin{assumption}\label{Assumption_2}
		Let $ {\rm supp} \, \widehat \psi\subset [-A,A],$ $A>0,$ and
		$\widehat \psi(\cdot)$ is of bounded variation on~$[-A,A].$
	\end{assumption}
	This assumption is technical and can be replaced by a sufficiently fast decay rate of $\widehat \psi(\cdot)$ at infinity. For example, the simulation studies in \cite[Example~3]{Olenko:2022} demonstrated that the decay rate of the Gaussian tail order is sufficient for the validity of the obtained results.
	
	To formulate the obtained results we will use the following constants
	\[c_2 := \int_\mathbb{R}|\widehat{\psi}(\lambda)|^2 d\lambda \quad \mbox{and}\quad c_3:=\max_{\lambda \in [-A, A]}|\widehat{\psi} (\lambda) |^2.\]
	
	\begin{definition}\label{def2} For any pair $(a, b) \in \mathbb{R}_+ \times \mathbb{R},$  the filter transform of the process $X(t)$ is defined by
		\begin{equation*}
			d(a, b) := \frac{1}{\sqrt{a}} \int_{\mathbb{R}} \psi\left( \frac{t - b}{a}\right) X(t)dt = \sqrt{a} \int_{\mathbb{R}} e^{ib\lambda}  \overline{\widehat{\psi}(a\lambda)}\sqrt{f(\lambda)}dW(\lambda).
		\end{equation*}
		
	\end{definition}
	Definition~\ref{def2} provides two equivalent expressions of the filter transform in the spectral and time domains. By the construction, $d(a, b)$ is a Gaussian random variable with
	\[\mathbb{E}d(a, b) = 0,\]
	\[\mathbb{E}|d(a, b)|^2 = a \int_{\mathbb{R}} |\widehat{\psi}(a\lambda)|^2 f(\lambda) d\lambda := J(a).\]
	
	A very detailed motivation, discussion,  and various particular examples of the introduced model, which include wavelet transforms and Gegenbauer processes as special important cases, can be found in publications~\cite{AAFO, Olenko:2022}.

	\bigskip

	\begin{assumption} \label{Assumption_3}
		Let the integral $ \int_{\mathbb{R}^2} \psi(t)B(t-t')\overline{\psi(t')}dt dt'$
		exists as an improper Cauchy integral on the plane. This integral always exists as a Lebesgue integral since $|B|$ is a bounded function and $\psi$ is an integrable function. Indeed, 
			\begin{align*}
				&\int_{\mathbb{R}^2} \big|\psi(t)B(t-t')\overline{\psi(t')}\big|dt dt'
    \le B(0) \bigg (\int_{\mathbb{R}} \big|\psi(t)\big|dt\bigg)^2<\infty
			\end{align*}
	\end{assumption}

\section{First Discrete Statistics}\label{section 1st statistic}

This section derives results analogous to Section~3~in~\cite{AAFO}, but for a discretely sampled process $X(t)$.

Let us define the following discretized and truncated versions of $d(a, b)$
\[{d}^{(\theta)} (a, b) := \frac{1}{\sqrt{a}} \int_{-\theta}^{\theta} \psi\left( \frac{t - b}{a}\right) X(t)dt,\]
\begin{equation}\label{djk}
d^{(\theta,\delta)}(a, b) := \frac{1}{\sqrt{a}}\sum_{l=[-\theta/\delta]}^{[\theta/\delta]}\left(\int_{\max(-\theta,l\delta)}^{\min(\theta,(l+1)\delta)}\psi\left(\frac{t - b}{a}\right)dt\cdot X(\delta l)\right),
\end{equation}
where $\theta,\delta \in \mathbb{R_+},  \delta \leq \theta$ and $[\cdot]$ denotes the integer part function. 

\begin{remark}
	The introduced transform $d^{(\theta,\delta)}(a, b)$ represents a discrete analogue of the filter transform ${d} (a, b)$ of functional data. To compute it, only a finite set of values of $X(t)$ sampled on a discrete grid with the resolution $\delta$ is required. In this case, the filter coefficients (discrete weights) are given by $\int_{\max(-\theta,l\delta)}^{\min(\theta,(l+1)\delta)}\psi\left(\frac{t - b}{a}\right)dt,$ $l=[-\theta/\delta],\dots,[\theta/\delta],$ and can be readily computed for a given filter function $\psi(\cdot).$
\end{remark}

Let us use the following  notations
\[d_{jk} := d(a_j, b_{jk}), \quad {d}^{(\theta)}_{jk} := d^{(\theta)} (a_j, b_{jk}),\quad {d}_{jk}^{(\theta,\delta)} := d^{(\theta,\delta)} (a_j, b_{jk}),\]
\[
{d}_{j.}^{(2)} := \frac{1}{m_j}\sum_{k=1}^{m_j}d_{jk}^{2}.
\]
where $\{a_j\} \subset \mathbb{R_+},$ $\{m_j\} \subset \mathbb{N}, j \in \mathbb{N},$ are unboundedly monotone increasing sequences, and $\{b_{jk} \}\in \mathbb{R}, j \in \mathbb{N},$ $k\in \{1,\dots,m_j\}$ is such that $b_{jk_1} \ne b_{jk_2}$ for all $j \in \mathbb{N}$ and $k_1 \ne k_2,$ see~\cite{AAFO}.

In the following we consider the sequences $\{b_{jk}\}$ and $\{\gamma_{j}\}\subset \mathbb{R_{+}}$ such that  for all $ j,k_{1}, k_{2}\in \mathbb{N}$ it holds
$|b_{jk_{1}}-b_{j k_{2}}|\geq|k_{1}-k_{2}| \gamma_{j}.$ Also, we will be using the following notation $\tilde{b}_{j}:=\max_{k\in \{1,\dots,m_j\}}{|b_{jk}}|.$

Analogously to ${d}_{j.}^{(2)}$  let us introduce the next average of squared discretised coefficients
\begin{equation}\label{first_discr_stat}
	{d}_{j.}^{(2, \theta, \delta)} := \frac{1}{m_j}\sum_{k=1}^{m_j}\left({d}_{jk}^{{(\theta,\delta)}}\right)^2.
\end{equation}

To estimate the parameters $s_0$ and $\alpha$ we will introduce two statistics that are based on the coefficients ${d}_{j.}^{(2, \theta, \delta)}.$  To study the behaviour of  ${d}_{j.}^{(2, \theta, \delta)}$ and these statistics we need the following auxillary results.
\begin{lemma} \label{tails for d}
	Let $|\psi(u)| \le C |u|^{-q},$ where $q > 1$.Then, there is a constant $C_1>0$ such that for $\theta > |b_{jk}|$  it holds
	\begin{equation*}
		\mathbb{E} \left({d}_{jk}^{(\theta)}  - d_{jk} \right)^2 \le C_1 \frac{a_j^{2q - 1}}{(\theta-|b_{jk}|)^{2q - 2}}.
	\end{equation*}
\end{lemma}

\begin{proof} It follows from the definitions of ${d}_{jk}^{(\theta)}$ and $ d_{jk}$ that
	\begin{equation*}
		\begin{aligned}[b]
			&\mathbb{E}  \left( {d}_{jk}^{(\theta)}  - d_{jk} \right)^2\\
			&\ = \mathbb{E} \left( \frac{1}{\sqrt{a_j}} \int_{-\infty}^{-\theta}   \psi \left(\frac{t - b_{jk}}{a_j}\right) X(t) dt + \frac{1}{\sqrt{a_j}} \int_{\theta}^{\infty}   \psi \left(\frac{t - b_{jk}}{a_j}\right) X(t) dt \right)^2 \\
			&\
			\le \frac{2}{a_j} \Bigg[ \underbrace{\mathbb{E}  \left( \int_{-\infty}^{-\theta}   \psi \left(\frac{t - b_{jk}}{a_j}\right) X(t) dt \right)^2 }_{I_1} + \underbrace{\mathbb{E}  \left( \int_{\theta}^{\infty}   \psi \left(\frac{t - b_{jk}}{a_j}\right) X(t) dt \right)^2}_{I_2}  \Bigg].
		\end{aligned}
	\end{equation*}

	Let us estimate the first expected value as
	\begin{equation*}
		\begin{aligned}[b]
			I_1 & = \int_{-\infty}^{-\theta} \int_{-\infty}^{-\theta}  \psi \left(\frac{t - b_{jk}}{a_j}\right)	\psi \left(\frac{s - b_{jk}}{a_j}\right)  \mathrm{B} (t -s) dt ds\\
			& \le \mathrm{B} (0) \left( \int_{-\infty}^{-\theta} \left|  \psi \left(\frac{t - b_{jk}}{a_j}\right)   \right|  dt \right) ^2 = a_j^2 \mathrm{B} (0) \left( \int_{-\infty}^{\frac{-\theta -b_{jk}}{a_j}} \left|  \psi \left(u\right)   \right|  du \right) ^2\\
			& \le C a_j^2 \mathrm{B} (0)  \left( \int_{-\infty}^{\frac{-\theta -b_{jk}}{a_j}} \left| u  \right|^{-q}  du \right) ^2\le C \frac{a_j^{2q}}{(\theta + b_{jk})^{2q-2}},   \end{aligned}
	\end{equation*}
	where the upper bound $C |u|^{-q}$ for $|\psi(u)|$ was used.
	
	Applying the same estimate approach to the second expected value $I_2$ one obtains the required upper bound
	\[\mathbb{E} \left( {d}_{jk}^{(\theta)}  - d_{jk} \right)^2 \le C \frac{a_j^{2q-1}}{(\theta + b_{jk})^{2q-2}} +  C \frac{a_j^{2q-1}}{(\theta - b_{jk})^{2q-2}} \le C_1 \frac{a_j^{2q-1}}{(\theta - |b_{jk}|)^{2q-2}}. \]
\end{proof}

\begin{remark}
	The assumption $|\psi(u)| \le C |u|^{-q}$ is satisfied for various filters. For example, by  the Paley-Wiener-Schwartz theorem  \cite[Theorem 7.2.2]{Strichartz}, if $\widehat \psi(\cdot)\in C^\infty$ has a compact support,  then for any $n\in \mathbb{N}$ there exists a constant $C_n$ such that $| \psi(u)|\le C_n (1+|u|)^{-n}.$
\end{remark}

The previous result implies the following lemma.
\begin{lemma}\label{lem6}
	Let the conditions of Lemma~{\rm\ref{tails for d}}  be satisfied, and for some $s\in (0,2]$ and $\delta_0>0$ it holds $\mathrm{B} (0)-\mathrm{B} (t)\le C_1t^s$  for all $t \in [0,\delta_0]$, where the constant $C_1$ only depends on $s$ and $\delta_0$. Then, there is a constant $C_2 > 0$ such that, for all $j\in\mathbb{N}$, for each $\delta\in [0,\delta_0]$, and for every $\theta >\max\{\tilde{b}_{j},\delta_0\}$, one has 
	\[\mathbb{E}\left({d}_{j.}^{(2, \theta, \delta)} - {d}_{j.}^{(2)} \right)^2\le  {C_2}  \left(\theta^2 \delta^{s} +  \frac{\theta a_j^{2q - 1}}{m_j}\sum_{k=1}^{m_j}\frac{1}{(\theta-|b_{jk}|)^{2q - 2}}\right),\]
    where $\tilde{b}_{j}=\max_{k\in \{1,\dots,m_j\}}{|b_{jk}}|.$
\end{lemma}

\begin{proof} 	Note, that $d_{jk}$ and~${d}_{jk}^{(\theta,\delta)}$ are Gaussian and $\mathbb{E}d_{jk} = \mathbb{E}{d}_{jk}^{(\theta,\delta)} = 0.$ Hence, by  Hölder's inequality and as $\mathbb{E}X^4 = 3\sigma^4$  for any Gaussian random variable $X \sim N(0, \sigma^2),$ it follows from the definitions of ${d}_{j.}^{(2, \theta, \delta)} $ and~${d}_{j.}^{(2)}$ that
	\begin{equation}
		\begin{aligned}[b]
			& \mathbb{E}\left({d}_{j.}^{(2, \theta, \delta)} - {d}_{j.}^{(2)} \right)^2= \mathbb{E} \left( \frac{1}{m_j} \sum_{k=1}^{m_j}\Big(\left({d}_{jk}^{( \theta, \delta)}\right)^2 - d_{jk}^2 \Big)\right)^2 \\
			&\quad\quad \le \frac{1}{m_j}\mathbb{E} \sum_{k=1}^{m_j}\left(\left({d}_{jk}^{( \theta, \delta)}\right)^2 - d_{jk}^2 \right)^2  =  \frac{1}{m_j}\mathbb{E}\left( \sum_{k=1}^{m_j}\left({d}_{jk}^{( \theta, \delta)} - d_{jk}\right)^2\right.\\
			&\quad\quad \times \left. \left({d}_{jk}^{( \theta, \delta)}+ d_{jk}\right)^2 \right) \le  \frac{1}{m_j} \sum_{k=1}^{m_j} \sqrt{\mathbb{E} \left({d}_{jk}^{( \theta, \delta)}- d_{jk}\right)^4 } \sqrt{\mathbb{E} \left({d}_{jk}^{( \theta, \delta)}+ d_{jk}\right)^4 } \\
			&\quad\quad =  \frac{3}{m_j} \sum_{k=1}^{m_j} \sqrt{\left(Var \left({d}_{jk}^{( \theta, \delta)}- d_{jk}\right) \right)^2 } \sqrt{\left(Var \left({d}_{jk}^{( \theta, \delta)}+ d_{jk}\right) \right)^2 } \\
			&\quad\quad = \frac{3}{m_j} \sum_{k=1}^{m_j} \mathbb{E} \left({d}_{jk}^{( \theta, \delta)}- d_{jk}\right)^2  \mathbb{E} \left({d}_{jk}^{( \theta, \delta)}+ d_{jk}\right)^2 \\
			&\quad\quad \le \frac{6}{m_j} \sum_{k=1}^{m_j}  \mathbb{E} \left({d}_{jk}^{( \theta, \delta)}- d_{jk}\right)^2  \left( \mathbb{E} \left({d}_{jk}^{(\theta,\delta)}\right)^2 + \mathbb{E} d_{jk}^2\right).\label{estim}
		\end{aligned}
	\end{equation}
	
	Notice that by Lemma~\ref{tails for d}
	\begin{equation}
		\begin{aligned}[b]
			\mathbb{E} \left({d}_{jk}^{(\theta,\delta)} - d_{jk}\right)^2 & = \mathbb{E} \Big( \left({d}_{jk}^{(\theta,\delta)} - {d}_{jk}^{(\theta)}  \right)   + \left({d}_{jk}^{(\theta)}  -  d_{jk}  \right)\Big)^2  \\
			& \le  2 \mathbb{E}  \left({d}_{jk}^{(\theta,\delta)} - {d}_{jk}^{(\theta)} \right) ^2 + 2 \mathbb{E}  \left( {d}_{jk}^{(\theta)}  -  d_{jk}  \right)^2  \\
			& \le 2  \mathbb{E}  \left({d}_{jk}^{(\theta,\delta)} - {d}_{jk}^{(\theta)} \right) ^2 + C_1 \frac{a_j^{2q - 1}}{(\theta-|b_{jk}|)^{2q - 2}}.\label{est1}
		\end{aligned}
	\end{equation}

		The first summand in (\ref{est1}) can be bounded as
				    \begin{align*}
				&\mathbb{E}  \left({d}_{jk}^{(\theta,\delta)} - {d}_{jk}^{(\theta)} \right) ^2 = \frac{1}{a_j}\mathbb{E}  \left(\sum_{l=[-\theta/\delta]}^{[\theta/\delta]}\int_{\max(-\theta,l\delta)}^{\min(\theta,(l+1)\delta)}\psi\left(\frac{t - b_{jk}}{a_j}\right)(X(t)-X(\delta l))dt\right)^2\\
				&\le \frac{C}{a_j}\left[\frac{\theta}{\delta}\right]\sum_{l=[-\theta/\delta]}^{[\theta/\delta]}\mathbb{E}\left(\int_{\max(-\theta,l\delta)}^{\min(\theta,(l+1)\delta)}\psi\left(\frac{t - b_{jk}}{a_j}\right)(X(t)-X(\delta l))dt\right)^2\\
				&\le \frac{C}{a_j}\left[\frac{\theta}{\delta}\right]\sum_{l=[-\theta/\delta]}^{[\theta/\delta]}\int_{\max(-\theta,l\delta)}^{\min(\theta,(l+1)\delta)}\psi^2\left(\frac{t - b_{jk}}{a_j}\right)dt\, \mathbb{E}\int_{\max(-\theta,l\delta)}^{\min(\theta,(l+1)\delta)} (X(t)-X(\delta l))^2dt\\
				&\le \frac{C}{a_j}\left[\frac{\theta}{\delta}\right]\int_{\mathbb{R}}\psi^2\left(\frac{t - b_{jk}}{a_j}\right)dt\, \int_{0}^{\delta} (\mathrm{B} (0)-\mathrm{B} (t))dt\le C\theta||\psi||^2 {\delta}^{s}.
  \end{align*}

		Thus, it follows from (\ref{est1})  that
		\begin{equation}\label{first}
			\mathbb{E} \left({d}_{jk}^{(\theta,\delta)} - d_{jk}\right)^2  \le   C\theta||\psi||^2 {\delta}^{s} + C_1 \frac{a_j^{2q - 1}}{(\theta-|b_{jk}|)^{2q - 2}}.
		\end{equation}
		
		Lemma~2 in~\cite{AAFO} gives the next upper bound of the $\mathbb{E} d_{jk}^2$ in (\ref{estim}),
		\begin{equation}\label{last} \mathbb{E} d_{jk}^2 = Cov\left( d_{jk}, d_{jk} \right) \le c_4(s_0, \alpha), \end{equation}
		where
		\[c_4(s_0, \alpha) := \frac{2}{s_0^{4\alpha}}  \max\left( \frac{2 c_2(1 + c_1)}{3}, c_3\, V_{-1/2}^{1/2} (h)+ (1 + c_1) \, V_{-A}^A\left(|\widehat{\psi} |^2\right) \right)\]
		and $V_a^b(f)$ denotes the total variation of a function $f(\cdot)$ on an interval $[a, b].$
		
		Applying H\"older's inequality one obtains the following upper bound for the $\mathbb{E} \left({d}_{jk}^{(\theta,\delta)}\right)^2$ in (\ref{estim})
		\begin{equation}
			\begin{aligned}[b]
				& \mathbb{E} \left({d}_{jk}^{(\theta,\delta)}\right)^2  = \frac{1}{a_j}\mathbb{E}  \left(\sum_{l=[-\theta/\delta]}^{[\theta/\delta]}\int_{\max(-\theta,l\delta)}^{\min(\theta,(l+1)\delta)}\psi\left(\frac{t - b_{jk}}{a_j}\right)X(\delta l)dt\right)^2\\
				&\quad \le \frac{C}{a_j}\left[\frac{\theta}{\delta}\right]\sum_{l=[-\theta/\delta]}^{[\theta/\delta]}\mathbb{E}\left(\int_{\max(-\theta,l\delta)}^{\min(\theta,(l+1)\delta)}\psi\left(\frac{t - b_{jk}}{a_j}\right)X(\delta l)dt\right)^2\\
				&\quad\le \frac{C}{a_j}\left[\frac{\theta}{\delta}\right]\sum_{l=[-\theta/\delta]}^{[\theta/\delta]}\int_{\max(-\theta,l\delta)}^{\min(\theta,(l+1)\delta)}\psi^2\left(\frac{t - b_{jk}}{a_j}\right)dt\, \int_{\max(-\theta,l\delta)}^{\min(\theta,(l+1)\delta)} \mathrm{B} (0)dt\\
				& \quad
				\le C\frac{\theta \mathrm{B} (0)}{a_j}\int_{\mathbb{R}}\psi^2\left(\frac{t - b_{jk}}{a_j}\right)dt\le C_3\theta||\psi||^2 . \label{seclast}
			\end{aligned}
		\end{equation}

		Now, by applying the upper bounds from (\ref{first}),  (\ref{last})  and (\ref{seclast})  to the terms in~(\ref{estim}) one obtains
		\begin{eqnarray}
			\mathbb{E}\left({d}_{j.}^{(2, \theta, \delta)} - {d}_{j.}^{(2)} \right) &\le &  \frac{6}{m_j}  \sum_{k=1}^{m_j}\left(C\theta||\psi||^2 \delta^{s} +   C_1\frac{a_j^{2q - 1}}{(\theta-|b_{jk}|)^{2q - 2}}\right) \nonumber \\
			&  &\times \left( c_4(s_0, \alpha) + C_3 \theta||\psi||^2 \right). \nonumber
		\end{eqnarray}

		Finally, the result of the lemma follows by combining the constants and replacing lesser terms with dominant ones.
	\end{proof}
	
	\begin{remark} Covariance functions satisfying the assumption of Lemma~\ref{lem6} have been widely used in the stochastic processes literature. Their origins can be traced back to the foundational publication \cite{Pit}. The asymptotic behaviour of the covariance function $\mathrm{B} (t)$ at the origin can be obtained from the asymptotic properties of the spectrum $f(\lambda)$ of $X(t)$ at the infinity, see, for example, Abelian and Tauberian theorems in \cite{Bing, LeonOle:2013}. It is important to note that these theorems require conditions on the asymptotic behaviour of the spectrum at infinity, but the considered model only specifies the singular behaviour of the spectrum at a fixed finite location. These two assumptions are unrelated and allow for the existence of a wide class of feasible processes.
	\end{remark}
	
	Let $\{\theta_j\},$  $\{\delta_j\},$ $\{r_{j}\}\subset \mathbb{R_{+}},$ $j\in \mathbb{N},$ be positive-valued sequences, such that   $\theta_j\to +\infty,$  $\delta_j\to 0,$  $r_{j}\to 0,$  when $j\to +\infty.$ Also we assume that $\{\theta_j\}$ is an increasing sequence, while $\{\delta_j\}$ and $\{r_{j}\}$ are decreasing sequences.
	
	The following result is the generalisation of \cite[Lemma~4]{AAFO} to the considered discrete case.
	
	\begin{lemma}\label{lemma_3}
		Let the conditions of Lemmas~{\rm\ref{tails for d} }and~{\rm\ref{lem6}} be satisfied with   $q>3/2$ and ${a_{j}}\geq {2A}$ for all $j\in\mathbb{N}.$   Let us choose such $\{m_{j}\}\subset \mathbb{N}$ that  \begin{equation}\label{cond1}
			\sum_{j=1}^{\infty} \frac{1}{r^{2}_{j}m_{j}}<+\infty\quad \mbox{and} \quad \displaystyle\sum_{j=1}^{\infty} \frac{a_{j}^{2}} {r_{j}^{2} \gamma_{j}^{2} m_{j} }<+\infty,
		\end{equation}
		and such $\{\theta_j\}$ and $\{\delta_j\}$       that
		\begin{equation}\label{cond2}
			\sum_{j=1}^{\infty} \frac{\theta_j^2 \delta_j^{s}}{r_{j}^{2}} < +\infty \quad \mbox{and} \quad   \sum_{j=1}^{\infty} \frac{\theta_j a_j^{2q - 1}}{r_{j}^2(\theta_j-\tilde{b}_{j})^{2q - 2}} < +\infty.
		\end{equation}
		Then, there exists an almost surely finite random variable $c_{5},$ such that for all $j\in \mathbb{N}$ it holds
		\[\left|{d}_{j.}^{(2, \theta_j, \delta_j)} -J(a_{j})\right|\leq c_{5}r_{j},\]
	   where 
			\[
			J(a_j)=\mathbb{E}|d(a_j, 0)|^2 = a_j \int_{\mathbb{R}} |\widehat{\psi}(a_j\lambda)|^2 f(\lambda) d\lambda.
			\]
	\end{lemma}
	
	\begin{proof}
		Using  Markov's inequality and \cite[Lemma~3]{AAFO}  we obtain
		\[ P\left(\left| {d}_{j.}^{(2, \theta_j, \delta_j)}  -J(a_{j})\right|>r_{j}\right) \leq \frac{\mathbb{E}\left({d}_{j.}^{(2, \theta_j, \delta_j)}  -J(a_{j})\right)^2}{r_{j}^{2}}\leq  \frac{2\mathbb{E}\left({d}_{j.}^{(2, \theta_j, \delta_j)}  -  {d}_{j.}^{(2)}\right)^2}{r_{j}^{2}}\]
		\[ +\frac{2\mathbb{E}\left( {d}_{j.}^{(2)} -J(a_{j})\right)^2}{r_{j}^{2}}\le \frac{C}{r_{j}^{2}m_j} \sum_{k=1}^{m_j} \left(\theta_j^2 \delta_j^{s} +  \frac{\theta_j a_j^{2q - 1}}{(\theta_j-|b_{jk}|)^{2q - 2}}\right)+ C\frac{\Big(1+\frac{\pi^{2}}{3}\frac{a_{j}^{2}}{\gamma_{j}^{2}}\Big)} { r_{j}^{2}m_{j}}\]
		\[ \le \frac{C}{r_{j}^{2}} \left(\theta_j^2 \delta_j^{s}+ \frac{\theta_j a_j^{2q - 1}}{(\theta_j-\tilde{b}_{j})^{2q - 2}}\right)+ C\frac{\Big(1+\frac{\pi^{2}}{3}\frac{a_{j}^{2}}{\gamma_{j}^{2}}\Big)} { r_{j}^{2}m_{j}}.\]
		
		By the choice of $\{m_{j}\},$ $\{\theta_j\}$ and $\{\delta_j\},$  and the conditions (\ref{cond1}) and (\ref{cond2})
		\[\displaystyle\sum_{j=1}^{\infty}P\left(\left|{d}_{j.}^{(2, \theta_j, \delta_j)}-J(a_{j})\right|>r_{_{j}}\right)<+\infty.\]
		Hence, the required statement follows by the Borel–Cantelli lemma.
	\end{proof}
	
	\begin{remark}\label{rem7}
		Notice that the conditions in (\ref{cond1}) coincide with those established in \cite{AAFO} and do not involve $\{\theta_j\}$ and $\{\delta_j\}.$ The detailed discussion and examples of sequences satisfying these conditions are provided in \cite{AAFO}.   In contrast, the conditions in (\ref{cond2}) encompass the sequences $\{\theta_j\},$ $\{\delta_j\},$ and $\{\tilde{b}_j\}.$ To ensure convergence of the second series in (\ref{cond2}), the sequence $\{\theta_j\}$ must have a considerably faster rate of increase than that of $\{\tilde{b}_j\}.$ Subsequently, an appropriate rate of decay for the sequence $\{\delta_j\}$ must be chosen to make the first series in (\ref{cond2}) finite.
	\end{remark}
	
	Now, we obtain the following analogue of \cite[Proposition~1]{AAFO} for the discrete-time sampling of $X(t).$
	
	\begin{theorem}\label{pro_1}
		Let the conditions of Lemma {\rm{\ref{lemma_3}}} be satisfied. Then, it holds that ${d}_{j.}^{(2, \theta_j, \delta_j)}\xrightarrow{a.s.} c_{2} s_{0}^{-4\alpha},$ when $j\to +\infty.$  Moreover, there exists an almost surely finite random variable $C$ such that  for all $j\in \mathbb{N}$
		\[
		\Big|{d}_{j.}^{(2, \theta_j, \delta_j)}- c_{2}s_{0}^{-4\alpha} \Big|\leq C\max(r_{j}, a_{j}^{-2}).
		\]
	\end{theorem}
	
	\begin{proof} The proof closely follows the steps outlined in \cite[Proposition~1]{AAFO} as  Lemma~\ref{lemma_3} established an analogous upper limit for the discrete version of the transform ${d}_{j.}^{(2, \theta_j, \delta_j)}$ as for its continuous counterpart in \cite{AAFO}. By applying Lemma~\ref{lemma_3} and the result from \cite[Lemma~5]{AAFO}, which states that
		for ${a_{j}}/{A}\geq2$ it holds $\left|J(a_{j})- c_{2} s_{0}^{-4\alpha}\right|\leq {C}/{a^2_{j}},$  one obtains the theorem's proposition.
	\end{proof}

	\section{Second Discrete Statistics}\label{section 2nd statistic}
	
	Results in this and the next sections are based on the results from Section~\ref{section 1st statistic} and are obtained analogously to \cite[Sections~4,~5]{AAFO}.  The estimates presented in Section~\ref{section 1st statistic} align with those in \cite{AAFO}, but have different constants and assumptions. Consequently, similar proofs were not included in this and the following sections to preclude redundancy. All notations have been adjusted to reflect the considered case, with comments provided where necessary to clarify the modifications.
	
	In this section, we introduce another statistics which allow us to estimate $ \alpha s_{0}^{-4\alpha -2}.$ Let us define 
	\begin{equation}\label{second_discr_stat}
		\Delta {d}_{j.}^{(2, \theta_j, \delta_j)}: =  \frac{{d}_{j.}^{(2, \theta_j, \delta_j)}-{d}_{j+1.}^{(2, \theta_{j+1}, \delta_{j+1})}} {a_{j}^{-2} -a_{j+1}^{-2}   }.
	\end{equation} 
	 Then the following result holds.
		
	\begin{theorem}\label{pro_2}
		Let the assumptions of Lemma {\rm \ref{lemma_3}} hold true and there exist $\varepsilon>0$ and $j_{0}\in \mathbb{N}$ such that $a_{j+1}\geq(1+\varepsilon) a_{j}$ for all $j\geq j_{0}.$  Then
		
		\begin{align*}
			\Delta {d}_{j.}^{(2, \theta_j, \delta_j)}  \xrightarrow{a.s.} \alpha  c_{3} s_{0}^{-4\alpha-2}, \quad j \to+\infty.
		\end{align*}
		Moreover, there exists an almost surely finite random variable $c_{6}$ such that  for all $j\in \mathbb{N}$ it holds 
		\[
			\Big| \Delta {d}_{j.}^{(2, \theta_j, \delta_j)} - \alpha  c_{3} s_{0}^{-4\alpha-2}  \Big| \leq   c_{6} \max\Big (a_{j}^{2} r_{j}, a_{j}^{-2}\Big). 
		\] 
	\end{theorem}	
	
	\begin{proof}
		The statement in the theorem is obtained in the same way as it is done in \cite[Proposition 2]{AAFO} by replacing $\Delta \bar{\delta}_{ j\cdot}^{(2)}$ with $\Delta {d}_{j.}^{(2, \theta_j, \delta_j)}$, $\bar{\delta}_{ j\cdot}^{(2)}$ with ${d}_{j.}^{(2, \theta_j, \delta_j)}$  and applying Lemma {\rm \ref{lemma_3}}.		
	\end{proof}

	\section{Estimation of \texorpdfstring{$(s_{0},\alpha)$}{soal}}\label{Section Estimation}

 This section combines results from \cite{AAFO} and the adjustment approach from \cite{Olenko:2022} to obtain simultaneous parameter estimators.
 
	Sections~\ref{section 1st statistic} and \ref{section 2nd statistic} proved that for the true values of parameters $(s_{0},\alpha)$ the vector of statistics
	\begin{align*}
		\begin{pmatrix} 
			{d}_{j.}^{(2, \theta_j, \delta_j)}/   c_{2} \\
			\Delta  {d}_{j.}^{(2, \theta_j, \delta_j)}/ c_{3}
		\end{pmatrix}
		\xrightarrow{a.s.} 
		\begin{pmatrix} 
			s_{0}^{-4\alpha} \\
			\alpha s_{0}^{-4\alpha-2}
		\end{pmatrix}
		,\quad j\to +\infty.  
	\end{align*}
	Analogously to the continuous case in \cite[Section 5]{AAFO}, let the pair $(\hat{s}_{0j}, \hat{\alpha }_{j})$ be a solution of the system of equations 
	\begin{align}\label{5_1}
		\begin{cases}
			\hat{s}_{0j}^{-4 \hat{\alpha }_{j}} = {d}_{j.}^{(2, \theta_j, \delta_j)}/   c_{2},\\
			\hat{\alpha }_{j}\hat{s}_{0j}^{-4\hat{\alpha }_{j}-2}=\Delta {d}_{j.}^{(2, \theta_j, \delta_j)}/ c_{3}.
		\end{cases}
	\end{align}

	There might be cases for which $\left(\frac{ {d}_{j.}^{(2, \theta_j, \delta_j)}}{   c_{2}}, \frac{\Delta {d}_{j.}^{(2, \theta_j, \delta_j)}} { c_{3}}\right) $ is not in the feasible region of $ \left( s_{0}^{-4 \alpha},   \alpha s_{0}^{-4\alpha -2} \right)$ in which the system (\ref{5_1}) has a unique solution. For this reason, we propose adjusted estimates.
	
	Let $(y_1,y_2 )\in \mathcal{D},$ where $\mathcal{D}:=\Big\{ (y_1,y_2) \in (0,1)\times (0,y^2_{1}/2) \Big\}.$ 
	It was shown in \cite[Lemma~8,~9]{AAFO} that the system
	\begin{align} 
		\begin{cases}\label{5_2}
			s_{0}^{-4 \alpha}=y_1,\\
			\alpha s_{0}^{-4\alpha -2}=y_2,
		\end{cases}
	\end{align}
	has a unique solution $(s_{0}, \alpha)\in  (1,+\infty) \times \big(0,\frac{1}{2}\big)$.\\
	Thus, if $\bigg(\frac{{d}_{j.}^{(2, \theta_j, \delta_j)}}{   c_{2}},\frac{\Delta{d}_{j.}^{(2, \theta_j, \delta_j)}}{ c_{3}}\bigg)\in \mathcal{D}$ then there is a pair $(\hat{s}_{0j}, \hat{\alpha }_{j})\in (1,+\infty)\times(0, \frac{1}{2})$ that satisfies the system of equations $( \ref{5_1}).$  
	
	The solution to \eqref{5_2} was given in \cite[Proposition~3]{AAFO} in terms of the $LambertW$ function, see  \cite{Corless1996}, by
	\begin{equation}
		(s_{0}, \alpha) = \Phi^{-1}(y_1,y_2):= \left(exp\left(\frac{G(y_1,y_2)}{2}     \right),\frac{y_2} {y_1} exp\left( G(y_1,y_2) \right)\right),\label{5_4}
	\end{equation} 
	where $G(y_1,y_2)=LambertW\left(-\frac{y_1\ln\left(y_1\right)} {2y_2}\right).$ The notation $\Phi^{-1}$ is used for consistency with notations in \cite{Olenko:2022}
	
	The system of equations {\rm(\ref{5_2})} has a unique solution given by (\ref{5_4}), where the branch Lambert$W_0$ of the function is used.

	Therefore, if $s_{0}$ and $\alpha$ are the true value of parameters then the corresponding $(y_1,y_2)\in \mathcal{D}.$ As $\mathcal{D}$ is an open set, there exists some $j_{0}\in \mathbb{N}$ such that $\bigg(\frac{{d}_{j.}^{(2, \theta_j, \delta_j)}}{   c_{2}},\frac{\Delta {d}_{j.}^{(2, \theta_j, \delta_j)}}{ c_{3}}\bigg)\in \mathcal{D}$  for all $j\geq j_{0}$ and system (\ref{5_1}) has a unique solution. However, for some $j<j_0$ it might happen that the estimated values of statistics $\bigg(\frac{{d}_{j.}^{(2, \theta_j, \delta_j)}}{   c_{2}},\frac{\Delta{d}_{j.}^{(2, \theta_j, \delta_j)}}{ c_{3}}\bigg)\notin \mathcal{D}$ even if   $(y_1,y_2)\in \mathcal{D}$ for the  corresponding true value $ (s_{0},\alpha).$  To deal with such cases, similar to the continuous case in \cite[Definition~2]{AAFO} and \cite[Section 5]{Olenko:2022}, we suggest an adjustment of $\bigg(\frac{{d}_{j.}^{(2, \theta_j, \delta_j)}}{   c_{2}},\frac{\Delta{d}_{j.}^{(2, \theta_j, \delta_j)}}{ c_{3}}\bigg)$ that always has its values in $\mathcal{D}.$
	
	For $\eps\in(0,1)$ and $(y_1,y_2)\in\mathbb{R}^2,$ let us define the following continuous vector-valued truncating function $\mathcal{T}$  taking values in $\mathcal{D}$
	\[\mathcal{T}(y_1,y_2,\eps) = \Big( \mathcal{T}_1(y_1,\eps) ~,~ \mathcal{T}_2(y_1,y_2,\eps) \Big) ,\]
	where
	\begin{align*}  \mathcal{T}_1(y_1,\eps) &:= \max( \eps , \min( y_1, 1-\eps) ) = \begin{cases}
			\eps,   & \text{ if } y_1 \le \eps, \\
			y_1,    & \text{ if } \eps \le y_1 \le 1-\eps, \\
			1-\eps, & \text{ if } y_1 > 1-\eps,
		\end{cases}  \\
		\mathcal{T}_2(y_1,y_2,\eps) &:= \max\left( {\eps^2}/{4} , \min\left( y_2, \frac{ \big(\mathcal{T}_1(y_1,\eps)\big)^2 }{2}-{\eps^2}/{4} \right) \right)\\
		&\ \ = \begin{cases}
			{\eps^2}/{4},   & \text{ if } y_2 \le {\eps^2}/{4}, \\
			y_2,    & \text{ if } {\eps^2}/{4} \le y_2 \le \frac{ \big(\mathcal{T}_1(y_1,\eps)\big)^2 }{2}-{\eps^2}/{4}, \\
			\frac{ \big(\mathcal{T}_1(y_1,\eps)\big)^2 }{2}-{\eps^2}/{4},
			& \text{ if } y_2 > \frac{ \big(\mathcal{T}_1(y_1,\eps)\big)^2 }{2}-{\eps^2}/{4}.
		\end{cases}
	\end{align*}
	Note, that $\mathcal{T}\left( \frac{{d}_{j.}^{(2, \theta_j, \delta_j)} }{ c_2 } ~,~
	\frac{ \Delta{d}_{j.}^{(2, \theta_j, \delta_j)}} { c_3} ~,~
	\varepsilon\right) \in \mathcal{D}$ and coincides with those values $\left( \frac{{d}_{j.}^{(2, \theta_j, \delta_j)} }{c_2} ~,~
	\frac{ \Delta{d}_{j.}^{(2, \theta_j, \delta_j)}} { c_3 } \right)$ that belong to $\mathcal{D}.$
	The geometric illustrations for this and analogous adjusted values were provided in \cite{AAFO} and \cite{Olenko:2022}. Also, it was shown that they have the same rate of convergence to $ \left( s_{0}^{-4 \alpha},   \alpha s_{0}^{-4\alpha -2} \right)$ as their non-adjusted counterparts.
	
	Analogously to \cite{AAFO, Olenko:2022} we introduce discrete versions of adjusted statistics of the parameters $s_0$ and $\alpha.$
	\begin{definition}\label{def_adjusted_stats}
		The adjusted statistic for the parameter $(s_0,\alpha)$ is
		\[
		\widehat{(s_0,\alpha)}_j := \Phi^{-1} \left(
		\mathcal{T}\left( \frac{{d}_{j.}^{(2, \theta_j, \delta_j)} }{c_2} ~,~
		\frac{ \Delta{d}_{j.}^{(2, \theta_j, \delta_j)}} {c_3} ~,~
		\frac{1}{m_j} \right) \right).
		\]
	\end{definition}


	Finally, the main result is given below.

	\begin{theorem}\label{Theorem_1} 
		Let the process $X(t)$ and the filter $\psi(\cdot)$ satisfy Assumptions~\mbox{\rm{\ref{Assumption_1}--\ref{Assumption_3}}} \it{and
			$(\widehat{s}_{0j}, \widehat{\alpha }_{j})$ be the adjusted statistic. 	
			If $s_{0}$ and $\alpha$ are the true values of parameters and the assumptions of Theorem~\rm{\ref{pro_2}} {\it hold true, then
				\[\widehat s_{0j} \xrightarrow{a.s.} s_{0}\quad \mbox{and}\quad \widehat \alpha_j \xrightarrow{a.s.}\alpha,\quad   j\to +\infty.\]
				Moreover, there are almost surely finite random variables $c_{5}$ and $c_{6}$ such that for all $j\in \mathbb{N}$ it holds
				\[
				|\widehat s_{0j}-s_{0}|\leq c_{5} \max\big (a_{j}^{2} r_{j}, a_{j}^{-2}\big)\
				\]
				\text{and} 
				\[|\widehat \alpha_{j}-\alpha|\leq c_{6} \max\big (a_{j}^{2} r_{j}, a_{j}^{-2}\big). 
				\]}}
	\end{theorem} 
	
	\begin{remark} Note that for any sequences $\{a_j\}$ and $\{r_j\}$ there exists a sequence $\{m_j\}$ such that the condition~(\ref{cond1}) holds true.  Thus, choosing a sufficiently fast increasing range of averaging over the index $k$ and increasing and decreasing rates of the sequences $\{\theta_j\}$ and $\{\delta_j\}$ respectively, by Remark~\ref{rem7}, one can obtain the desirable convergence rate at levels $j.$ 
	\end{remark}

	\section{Numerical studies}\label{sec_num} 
	In this section, the results are illustrated and validated through simulation modelling.
 
 First, we generated realisations of the Gegenbauer-type time series and calculated the considered statistics ${d}_{j.}^{(2, \theta_j, \delta_j)}$ and $\Delta{d}_{j.}^{(2, \theta_j, \delta_j)}$ using the Mexican hat wavelet as a filter. Next, we computed the statistics $\widehat{s}_{0j}$ and $\widehat{\alpha}_{j}$ using the established formulas. This procedure was repeated multiple times for various numbers of observations and different levels of $j$ to numerically investigate empirical distributions of the estimators.  Finally, we demonstrated the convergence of the estimators to their theoretical values as the number of observations and level $j$ increased.
	
The following representation of the Gegenbauer process $X(t)$ was employed, see \cite[p. 47]{Leonenko:1999}
	\begin{equation} \label{disrete_extension}
		X(t) = \sum_{n=0}^{\infty } C_n^\mu (2\eta) \varepsilon (t - n),
	\end{equation}
	where $C_n^\mu(\cdot)$ is the Gegenbauer polynomial, $\varepsilon(t), t \in \mathbb{Z},$ is a zero-mean white noise with the variance $\sigma_{\varepsilon}^2.$ The Gegenbauer polynomials are defined by 
	\[ C_n^\mu(2\eta) = \sum_{k=0}^{[\frac{n}{2}]} (-1)^k  \frac{(2\eta)^{n - 2k}\Gamma(\mu + n - k)}{\Gamma(\mu) \Gamma(k + 1) \Gamma(n - 2k + 1)}, \]
	where $\Gamma(\cdot)$ is the Gamma function. Note, that
		\[ C_n^\mu(2\eta) \sim \frac{\cos((n + \mu)\nu -  (\mu\pi / 2))}{\Gamma(\mu) \sin^\mu (\nu)}\left( \frac{2}{n} \right)^{1-\mu},\quad \mbox{as}\ n \to \infty, \]
	where $\nu = \arccos \eta.$
	
	We generated the time series using the truncated version of \eqref{disrete_extension} with parameter values $\mu = 0.1$ and $\eta = 0.3.$ These values correspond to the parameters  $\alpha=0.1$ and $s_0 = \arccos(\eta) \approx 1.27,$ which are in the admissible region $\mathcal{D}.$ The plots of the corresponding spectral density function and a realization of this time series are shown in Figure~ \ref{fig1*}. Two cases were considered when the series in \eqref{disrete_extension} was truncated to $10^4$ and $10^7$ terms. 
 
 The Mexican hat wavelet, that was used to compute ${d}_{jk}^{(\theta,\delta)},$ is defined by
	 \[ \psi(t)=\frac{2}{\sqrt{3\sigma}\pi^\frac{1}{4} }\left( 1-\left(\frac{t}{\sigma}\right)^2\right) \mathrm{e}^{-\frac{t^2}{2\sigma^2}},\]    
	 where $\sigma > 0.$ Though, its Fourier transform 
	 \[\widehat \psi(\lambda) =\frac{\sqrt{8} \pi^\frac{1}{4} \sigma^\frac{5}{2}} {\sqrt{3}} \lambda^2 \mathrm{e}^{-\frac{\sigma^2 \lambda^2}{2}}\] does not have finite support, it exhibits a sufficiently fast decay rate  as~\mbox{$\lambda \to +\infty.$} As only finite sums and intervals are used in numerical computations, this guarantees the fulfilment of the assumptions. Additionally, the numerical studies suggest that the assumptions of finite support can be relaxed. The R package WMTSA was used to calculate ${d}_{jk}^{(\theta,\delta)}.$
	 
	The time series were generated 1,000 times, and for each simulation, the values for  ${d}_{j.}^{(2, \theta_j, \delta_j)},$ $\Delta{d}_{j.}^{(2, \theta_j, \delta_j)},$ $\widehat{s}_{0j}$  and $\widehat{\alpha }_{j}$ were computed. The process was repeated at different levels $j$ and fractions, where only shorter time series corresponding to the specified percentage of observations were used instead of the entire time series. The algorithm consisted of the following steps:
	 
	 \begin{itemize}
 \item compute ${d}_{jk}^{(\theta,\delta)}$ using (\ref{djk});
	 	\item 
	 	obtain the value for ${d}_{j.}^{(2, \theta_j, \delta_j)}$ using \eqref{first_discr_stat};	 	
	 	\item 
	 	calculate the value for $\Delta{d}_{j.}^{(2, \theta_j, \delta_j)}$ using \eqref{second_discr_stat};	 	
	 	\item 
	 	compute the estimates  $\widehat{s}_{0j}$  and $ \widehat{\alpha }_{j}$ using Definition~\ref{def_adjusted_stats}.
	 \end{itemize}
	 	In the simulations, the following values were used: $a_j = j,$ $b_{jk} = k,$ $\gamma_{j} = 1,$ $r_j = a_j^{-2.5},$ $m_j = a_j^9,$ $\theta_j = j^{13/6},$ and $\delta_j = j^{-22-{1}/{6}},$ $j = 1, \dots, 7.$  These values satisfy the assumptions of Theorem~\ref{pro_1}, see \cite[Example 4]{AAFO}. 
	
		\begin{figure}[!hbt]
		\begin{subfigure}{0.5\textwidth}
			\centering
			\includegraphics[width = \textwidth, height = \textwidth, trim={0 1.5cm 0 2cm},clip]{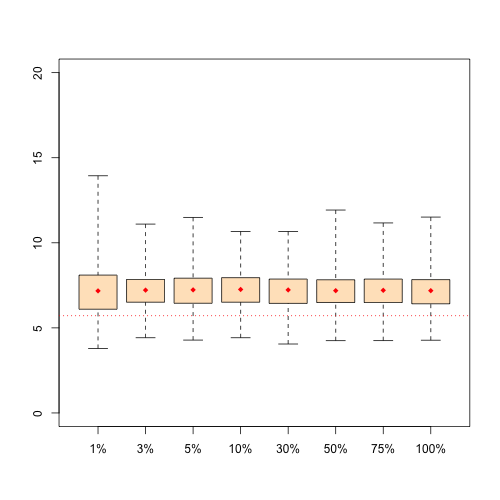}
			\caption{The case of $j=1$}
		\end{subfigure}
		\begin{subfigure}{0.5\textwidth}
			\centering
			\includegraphics[width = \textwidth, height = \textwidth, trim={0 1.5cm 0 2cm},clip]{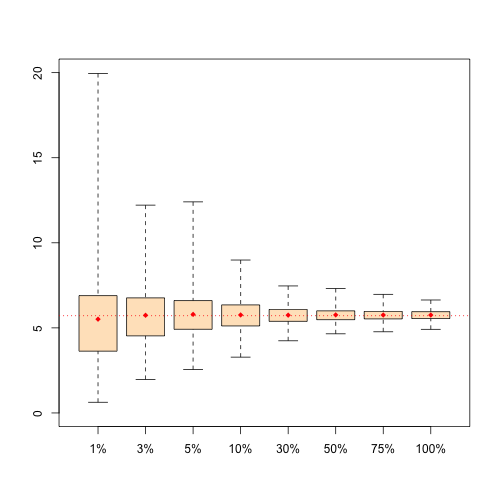}
			\caption{The case of $j=7$}
		\end{subfigure}
		\caption{Box plots of the ${d}_{j.}^{(2, \theta_j, \delta_j)},$ where 100\% corresponds to $10^4$ observations} \label{fig_10k_1stat_j17}
	\end{figure}
	\begin{figure}[!hbt]
		\begin{subfigure}{0.5\textwidth}
			\centering
			\includegraphics[width = \textwidth, height = \textwidth, trim={0 1.5cm 0 1.5cm},clip]{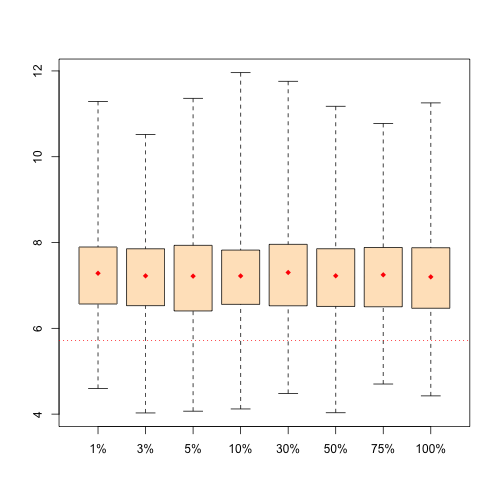}
			\caption{The case of $j=1$}
		\end{subfigure}
		\begin{subfigure}{0.5\textwidth}
			\centering
			\includegraphics[width = \textwidth, height = \textwidth, trim={0 1.5cm 0 1.5cm},clip]{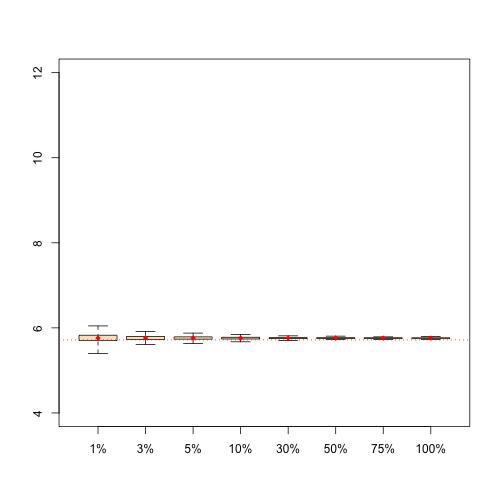}
			\caption{The case of $j=7$}
		\end{subfigure}
		\caption{Box plots of the ${d}_{j.}^{(2, \theta_j, \delta_j)},$ where 100\% corresponds to $10^7$  observations} \label{fig_10m_1stat_j17}
	\end{figure}

	 The box plots in Figures \ref{fig_10k_1stat_j17} and \ref{fig_10m_1stat_j17} illustrate the first statistic ${d}_{j.}^{(2, \theta_j, \delta_j)}$ across varying numbers of observations and levels of $j.$ In Figure~\ref{fig_10k_1stat_j17}, simulated time series consisting of $10^4$ values  were used. The horizontal axis denotes fractions of used observations, where 100\% represents the scenario using all values, and other percentages indicate subsets of observations used for computations (e.g., 1\% utilizes 100~simulated values). The left plot displays box plots for $j=1$, while the right plot corresponds to $j=7.$ The true values of the considered parameters are shown by horizontal dashed lines. Similar results are presented in Figure~\ref{fig_10m_1stat_j17} but for simulated time series with $10^7$ values. 
  
  The plots in Figures \ref{fig_10k_1stat_j17} and \ref{fig_10m_1stat_j17} reveal that for small values of $j,$ the estimated values of ${d}_{7.}^{(2, \theta_7, \delta_7)}$  can significantly diverge from the true values, resulting in a wide spread of estimates. However, consistent with theoretical findings, as both~$j$ and the number of observations increase, the estimates converge towards their theoretical values.	Similar numerical results were obtained for other statistics, but are not presented here due to space constraints.
		
	\begin{figure}[!htb]
		\begin{subfigure}{0.5\textwidth}
			\centering
			\includegraphics[width = \textwidth, height = \textwidth, trim={0 1.5cm 0 1.5cm},clip]{delta1_mat1_10000000_j_7.jpg}
			\caption{Case of ${d}_{7.}^{(2, \theta_7, \delta_7)}$}
		\end{subfigure}
		\begin{subfigure}{0.5\textwidth}
			\centering
			\includegraphics[width = \textwidth, height = \textwidth, trim={0 1.5cm 0 1.5cm},clip]{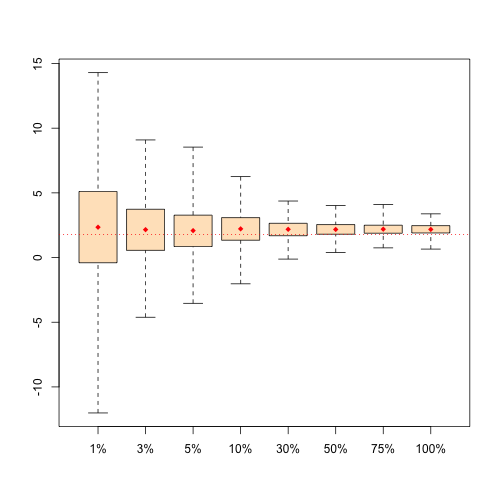}
			\caption{Case of $\Delta{d}_{6.}^{(2, \theta_6, \delta_6)}$}
		\end{subfigure}
		\begin{subfigure}{0.5\textwidth}
			\centering
			\includegraphics[width = \textwidth, height = \textwidth, trim={0 1.5cm 0 1.5cm},clip]{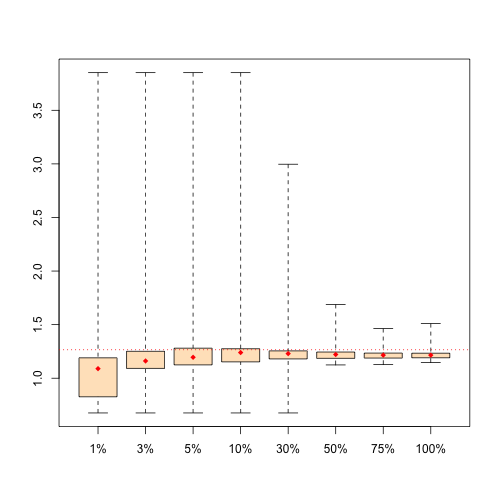}
			\caption{Case of $\widehat{s}_{06}$}
		\end{subfigure}
		\begin{subfigure}{0.5\textwidth}
			\centering
			\includegraphics[width = \textwidth, height = \textwidth, trim={0 1.5cm 0 1.5cm},clip]{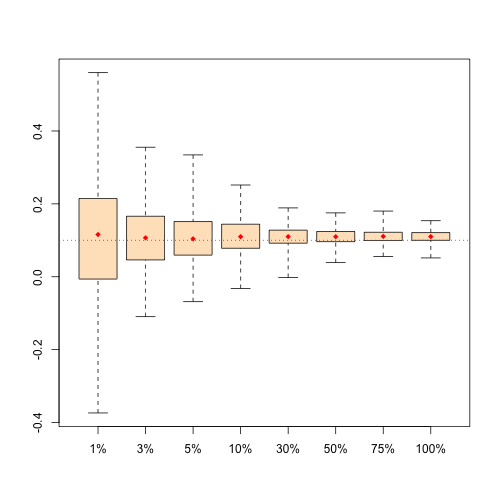}
			\caption{Case of $\widehat{\alpha }_{6}$}
		\end{subfigure}		
		\caption{Box plots of the statistics, where 100\% corresponds to $10^7$ of observations} \label{fig_10m_all_j67}
	\end{figure}
	
		Figure \ref{fig_10m_all_j67} illustrates the convergence of all estimators using the time series with~$10^7$ sampled values and their varying fractions. Given a large number of observations and the dense grid of time points, the 100\% fraction can be considered as the case of a functional time series, for which the closeness of the estimators to the true values was previously confirmed in \cite{AAFO, Olenko:2022}. Smaller percentages and numbers of sampled values correspond to discrete scenarios with shorter observation periods and lower time densities of observations. Increasing the fraction with $j$ entails increasing both the length of the truncated sampling time interval (controlled by $\theta_j$) and the sampling time rate (given by $\delta_j^{-1}$).
  
Compared to Figures~\ref{fig_10k_1stat_j17} and \ref{fig_10m_1stat_j17}, rather accurate estimates for the case of ${d}_{7.}^{(2, \theta_7, \delta_7)}$ are obtained even for 1\% of the simulated values. According to Theorems~\ref{pro_1},~\ref{pro_2} and~\ref{Theorem_1}, the rates of convergence for $\Delta{d}_{j.}^{(2, \theta_j, \delta_j)},  \widehat{s}_{0j} \mbox{ and } \widehat{\alpha }_{j}$ are different and slower than for ${d}_{j.}^{(2, \theta_j, \delta_j)}.$  Figure~\ref{fig_10m_all_j67} confirms this theoretical result. It also demonstrates that when the fraction increases, the estimators quickly converge to their functional counterparts, and therefore, to the true values of the parameters. Additionally, the simulation results demonstrate that it is enough to use the level $j=7,$ as the estimates of the parameters $s_0$ and $\alpha$ are close to their theoretical values starting from the 5\% fraction, with the recommended values being around 30-50\%.

	\begin{figure}[!htb]
		\centering
		\includegraphics[width = \textwidth, height = 0.6\textwidth, trim={0 0 0 1.5cm},clip]{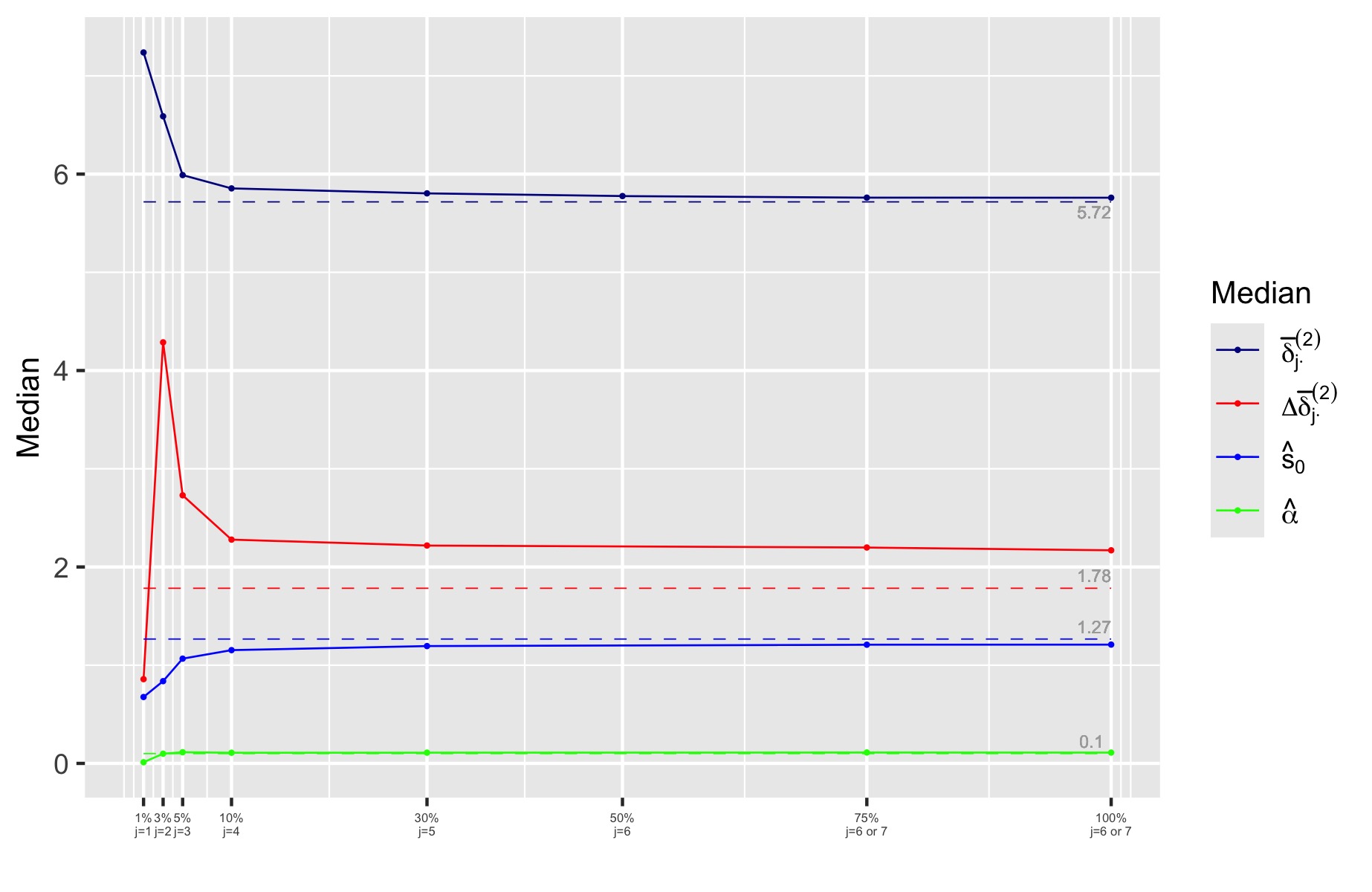}
		\caption{Medians of estimates for increasing fractions and $j$} \label{fig_stats_medium}
	\end{figure}
	
	\begin{figure}[!htb]
		\centering
		\includegraphics[width = \textwidth, height = 0.6\textwidth, trim={0 0 0 1.5cm},clip]{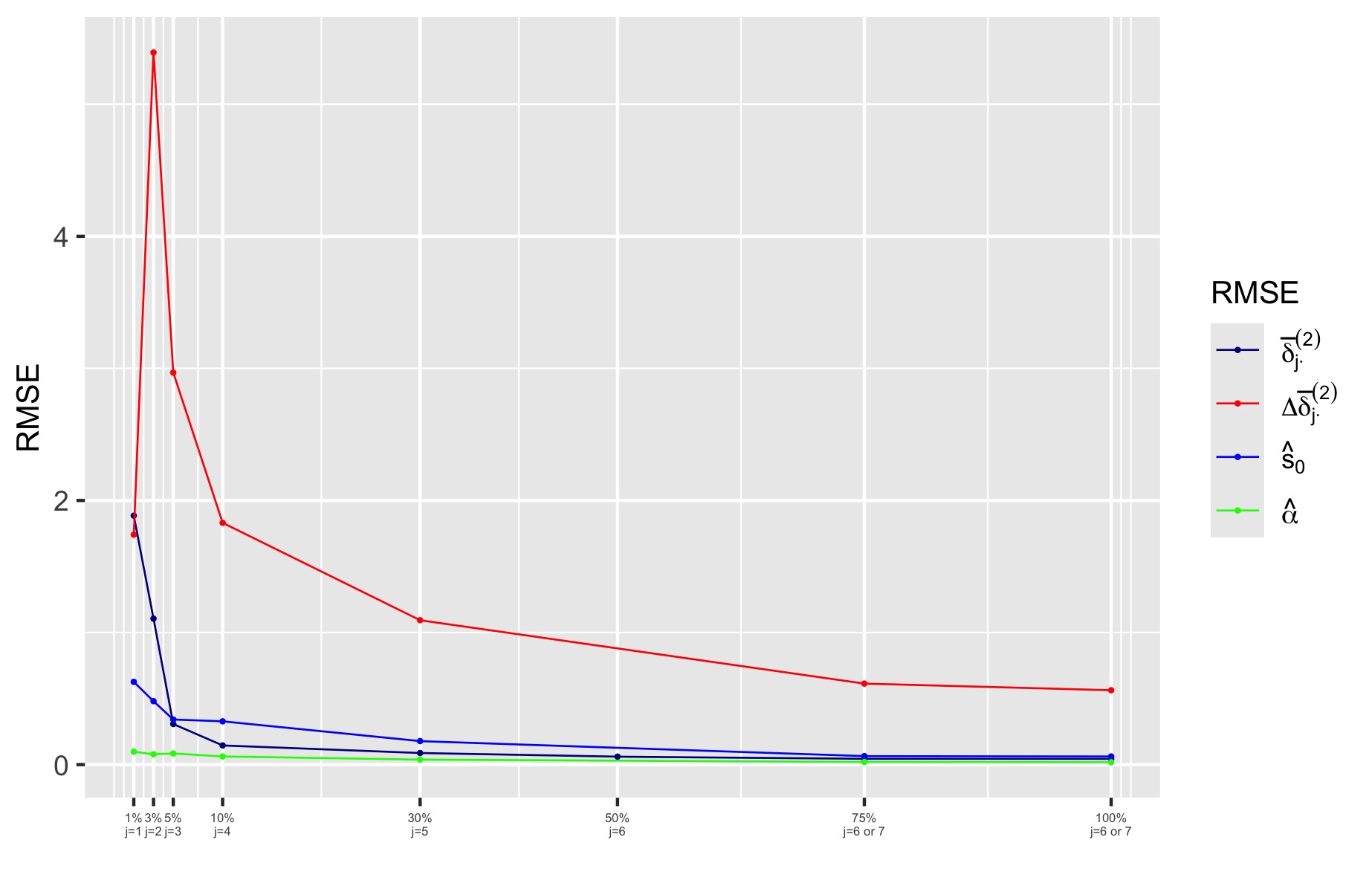}
		\caption{RMSE for increasing fractions and $j$} \label{sample_sizes_vs_mse}
	\end{figure}
	
	Figure \ref{fig_stats_medium} illustrates the convergence as both the level, the number and sampling rate of simulated values (that correspond to $j,$  $\theta_j$ and $\delta_j^{-1}$) increase. The corresponding root mean squared errors (RMSE) decrease, as shown in Figure~\ref{sample_sizes_vs_mse}. The plots confirm the converges of median values of the estimators to the true values of parameters, and also their mean square convergence. Initially, as expected, the deviations are chaotic and quite significant. However, as the number of levels and sampling intervals increases, the estimators converge to the true values. The rate of convergence of the second statistic, $\Delta {d}_{j.}^{(2, \theta_j, \delta_j)},$ is noticeably slower compared to the other three statistics. Nevertheless, given the rapid convergence of the first statistic, it is evident that the studied adjusted estimators $\widehat{s}_{0j}$ and $\widehat{\alpha }_{j}$  also exhibit good convergence rates. Figure~\ref{sample_sizes_vs_mse}  suggests that the recommended values are approximately 4-5 for $j$ and 30-50\% for the fraction. 

	\section{Conclusion}\label{concl}

	The paper investigates a semiparametric model of cyclic long-memory time series.  It demonstrates that the generalized filtered method-of-moment approach, initially developed for functional data, can be expanded to accommodate the discretely sampled scenarios. The derived estimates exhibit analogous properties to the continuous case, in particular, strong consistency. The obtained results on the closeness of discrete and continuous filter transforms have potential applications in other areas of statistics and stochastics.
	
	Some important areas for future studies include:
	
	\begin{itemize}
		\item[--] Further investigating asymptotic distributions of the considered discrete estimates, see~\cite{Alodat:2022, Olenko:2022};
		\item[--] Extending the proposed approach to the case of time series sampled at random locations;
		\item[--] Extending the approach to the spatial case, see~\cite{Ayache:2018, Espejo:2015, Klyolen:2012};
		\item[--] Further studying properties of filter transforms, in particular, their applications to nonlinear transformations of cyclic long memory time series, see~\cite{Alodat:2020}.
		\item[--] Comparing properties of the proposed discrete statistics with least squares, likelihood-type, and other existing methods via simulation studies, see~\cite{Smallwood:2021, Hunt:2022, Hunt:2023, Whitcher:2004}.
	\end{itemize}
	\section*{Acknowledgements} This research was supported under the Australian Research Council's Discovery Projects funding scheme (project number  DP220101680).  Antoine Ayache was partially supported by the Labex CEMPI (ANR-11-LABX-0007-01). Andriy Olenko was also partially supported by La Trobe University's SCEMS CaRE and Beyond grant.

	\bibliographystyle{abbrv}
	\bibliography{ref}
	
\end{document}